\documentclass[12pt]{article}
\usepackage[colorlinks]{hyperref}
\usepackage{color}
\usepackage{graphicx}
\usepackage{graphics}
\usepackage{makeidx}
\usepackage{showidx}
\usepackage{latexsym}
\usepackage{amssymb}
\usepackage{verbatim}
\usepackage{amsmath}
\usepackage{amsthm}
\usepackage{amsfonts}
\usepackage{anysize}
\usepackage{setspace}
\marginsize{3 cm}{3 cm}{3 cm}{3 cm}
\newtheorem{Theorem}{Theorem}[section]
\newtheorem{cor}{Corollary}[section]

\title{The Warped Product of Hamiltonian Spaces}
\author{Hassan Attarchi$^1$, Morteza Mirmohammad Rezaii$^2$}

\begin{document}
\maketitle
\noindent
\begin{abstract}
In this paper, the geometric properties of warped product Hamiltonian spaces are studied. It is shown there is a close geometrical relation between a warped product Hamiltonian space and its base Hamiltonian manifolds. For example, it is proved that for non-constant warped function $f$, the Sasaki lifted metric $G$ of Hamiltonian warped product space is bundle-like for its vertical foliation if and only if based Hamiltonian spaces are pseudo-Riemannian manifolds.
\vspace{.5cm}

\end{abstract}

\let\thefootnote\relax\footnotetext{$^{1,2}$Department of Mathematics and Computer Science, Amirkabir University of Technology (Tehran Polytechnic), Tehran, Iran.}
\footnote{$^1$Corresponding author, Phone number: +98 911 1751603\\ Email: \ hassan.attarchi@aut.ac.ir; hassan.attarchi@gmail.com}

\section{Introduction}
The notion of warped product spaces was introduced to study manifolds with negative curvatures by Bishop and O'Neill~\cite{bi on}. Afterwards, warped product was used to model the standard space-time, especially in the neighborhood of stars and black holes~\cite{onil}. The notion of warped product Finslerian manifolds was initially introduced by Kozma~\cite{kozm} in $2001$. Recently, it was developed by one of the present authors~\cite{Rezaii,hoosh,rez}. In this work, warped product of Hamiltonian spaces is introduced and it is shown that these spaces obtain Hamiltonian structure as well. Moreover, some geometric properties of warped product Hamilton spaces such as its nonlinear connections are studied.

Lagrange space has been certified as an excellent model for some important problems in Relativity, Gauge Theory and Electromagnetism~\cite{anastasia,anastasia1}. The geometry of Lagrange spaces gives a model for both the gravitational and electromagnetic field. Moreover, this structure plays a fundamental role in study of the geometry of tangent bundle $TM$. The geometry of the cotangent bundle $T^*M$ and the tangent bundle $TM$ which follows the same outlines are related by Legendre transformation. From this duality, the geometry of a Hamiltonian space can be obtained from that of certain Lagrangian space and vice versa. Using this duality several important results in the Hamiltonian spaces can be obtained: the canonical nonlinear connection, the canonical metrical connection, etc. Therefore, the theory of Hamiltonian spaces has the same symmetry and beauty as Lagrangian geometry. Moreover, it gives a geometrical framework for the Hamiltonian theory of Mechanics or Physical fields. With respect to the importance of these spaces in Physical areas, present work is formed to develop the concept of warped product on Hamiltonian spaces. Aiming at our purpose, this paper is organized in the following way:

Let $(M,H)$ be the warped Hamiltonian space of Hamiltonian spaces $(M_1,H_1)$ and $(M_2,H_2)$. In section 2, the notion of warped product Hamiltonian spaces is presented and some natural geometrical properties of cotangent bundle for a warped manifold are given. In section 3, it is shown that $(M,H)$ is a Hamiltonian space and its canonical nonlinear connections are calculated as well. Moreover, Sasakian lifted metric $G$ on $T^*M$ is introduced. In section 4, the Levi-Civita connection of pseudo-Riemannian metric $G$ on $T^*M$ is calculated. Finally in section 5, we prove some theorems that they show close relation between the geometry of a warped product Hamiltonian manifolds and their base Hamiltonian spaces.

\section{Preliminaries and Notations}
Here, a Hamiltonian space is a pair $(M,H)$, where $M$ is a real $n$-dimensional manifold and $H:T^*M\longrightarrow I\!\!R$ is a smooth function which its Hessian with respect to the cotangent bundle coordinate is a $d$-tensor field of type $(2,0)$ symmetric, nondegenerate and of constant signature on $T^*M\backslash\{0\}$. Let $\mathbb{H}_1^n=(M_1,H_1)$ and $\mathbb{H}_2^m=(M_2,H_2)$ be two Hamiltonian spaces with $\dim(\mathbb{H}_1^n)=n$ and $\dim(\mathbb{H}_2^m)=m$, respectively. The warped product of these spaces is denoted by $\mathbb{H}=(M,H)$ where:
\begin{equation}~\label{w p H}
M=M_1\times M_2\ \ \ \ \ and \ \ \ \ \ \ H=H_1+fH_2,
\end{equation}
for some smooth function $f:M_1\longrightarrow\mathbb{R}^+$. Then a coordinate system on $M$ is denoted by $\{(U\times V,\varphi\times\psi)\}$, where $\{(U,\varphi)\}$ and $\{(V,\psi)\}$ are coordinate systems on $M_1$ and $M_2$, respectively, such that each $\mathbf{x}=(x,z)\in M$ has the local expression $(x^i,z^{\alpha})$. It is notable that throughout the paper, the indices $\{i,j,k,...\}$ and $\{\alpha,\beta,\lambda,...\}$ are used for the ranges $1,...,n$ and $1,...,m$, respectively. Moreover, the canonical projections of $T^*M_1$ on $M_1$ and $T^*M_2$ on $M_2$ are denoted by $\pi_1$ and $\pi_2$, respectively. The fibre of cotangent bundle at $\mathbf{x}=(x,z)\in M$ is $T_{(x,z)}^*M=T_x^*M_1\oplus T_z^*M_2$, therefore $T^*M=T^*M_1\oplus T^*M_2$.

The induced coordinate system on $T^*M_1$ and $T^*M_2$ are $(x^i,p_i)$ and $(z^{\alpha},q_{\alpha})$, respectively, which the coordinate $p_i$ and $q_{\alpha}$ are called \emph{momentum variables}~\cite{miron}. The change of these coordinates on $T^*M_1$ and $T^*M_2$ are given by
\begin{equation}~\label{change}
\begin{array}{l}
\left\{
\begin{array}{l}
\tilde{x}^i=\tilde{x}^i(x^1,...,x^n),\cr \text{rank}\left(\frac{\partial\tilde{x}^i}{\partial x^j}\right)=n,\cr
\tilde{p}_i=\frac{\partial x^j}{\partial\tilde{x}^i}p_j
\end{array}
\right.\ \ \ \ \left\{
\begin{array}{l}
\tilde{z}^{\alpha}=\tilde{z}^{\alpha}(z^1,...,z^m),\cr \text{rank}\left(\frac{\partial\tilde{z}^{\alpha}}{\partial z^{\beta}}\right)=m,\cr
\tilde{q}_{\alpha}=\frac{\partial z^{\beta}}{\partial\tilde{z}^{\alpha}}q_{\beta}
\end{array}
\right.
\end{array}
\end{equation}
Let $(\mathbf{x},\mathbf{p})=(x,z,p,q)\in T^*M=T^*M_1\oplus T^*M_2$, the tangent space at $(\mathbf{x},\mathbf{p})$ to $T^*M$ is denoted by $T_{(\mathbf{x},\mathbf{p})}T^*M$ that is a $2(n+m)$-dimensional vector space. The natural basis induced on $T_{(\mathbf{x},\mathbf{p})}T^*M$ by local coordinate of $T^*M_1$ and $T^*M_2$ is $\{\frac{\partial}{\partial x^i},\frac{\partial}{\partial z^{\alpha}},\frac{\partial}{\partial p_i},\frac{\partial}{\partial q_{\alpha}}\}$. These coordinates are changed with respect to transformations~(\ref{change}) as follows:
\begin{equation}~\label{change1}
\left\{
\begin{array}{l}
\frac{\partial}{\partial x^i}=\frac{\partial\tilde{x}^j}{\partial x^i}\frac{\partial}{\partial\tilde{x}^j}+\frac{\partial\tilde{p}_j}{\partial x^i}\frac{\partial}{\partial\tilde{p}_j},\cr
\frac{\partial}{\partial z^{\alpha}}=\frac{\partial\tilde{z}^{\beta}}{\partial z^{\alpha}}\frac{\partial}{\partial\tilde{z}^{\beta}}+\frac{\partial\tilde{q}_{\beta}}{\partial z^{\alpha}}\frac{\partial}{\partial\tilde{q}_{\beta}},\cr
\frac{\partial}{\partial p_i}=\frac{\partial x^i}{\partial\tilde{x}^j}\frac{\partial}{\partial\tilde{p}_j},\cr
\frac{\partial}{\partial q_{\alpha}}=\frac{\partial z^{\alpha}}{\partial\tilde{z}^{\beta}}\frac{\partial}{\partial\tilde{q}_{\beta}}.
\end{array}
\right.
\end{equation}
In this work, the notations $\dot{\partial}^i$ and $\dot{\partial}^{\alpha}$ are used instead of $\frac{\partial}{\partial p_i}$ and $\frac{\partial}{\partial q_{\alpha}}$, respectively, similar to the notations in~\cite{miron}. The Jacobian matrix of transformations~(\ref{change1}) is
\begin{equation}~\label{cha.mat.}
\text{Jac}:=\left(%
\begin{array}{cccc}
  \frac{\partial\tilde{x}^j}{\partial x^i} & 0 & 0 & 0\\
  0 & \frac{\partial\tilde{z}^{\beta}}{\partial z^{\alpha}} & 0 & 0\\
  \frac{\partial\tilde{p}_j}{\partial x^i} & 0 & \frac{\partial x^i}{\partial\tilde{x}^j} & 0\\
  0 & \frac{\partial\tilde{q}_{\beta}}{\partial z^{\alpha}} & 0 & \frac{\partial z^{\alpha}}{\partial\tilde{z}^{\beta}}
\end{array}%
\right).
\end{equation}
It follows
$$\det(\text{Jac})=1.$$
By means of last equation, we have the following corollary.
\begin{cor}~\label{orient}
The manifold $T^*M=T^*M_1\oplus T^*M_2$ is orientable.
\end{cor}
Let $\bar{\partial}^a$ and $\frac{\partial}{\partial\mathbf{x}^a}$ be abbreviations for $\dot{\partial}^i\delta_i^a+\dot{\partial}^{\alpha}\delta_a^{\alpha+n}$ and $\frac{\partial}{\partial x^i}\delta_a^i+\frac{\partial}{\partial z^{\alpha}}\delta_{\alpha+n}^a$, respectively, where the indices $\{a,b,c,...\}$ are used for the range $1,...,n+m$. Throughout the paper, these notations and range of the indices are established.

We know that there are some natural structures live on the cotangent bundle $T^*M$, it would be interesting to present them on cotangent bundle of a warped product Hamiltonian space. First, the \emph{Liouville-Hamilton vector field} of $T^*M$ is given by:
\begin{equation}~\label{l h v}
C^*:=\mathbf{p}_a\bar{\partial}^a=p_i\dot{\partial}^i+q_{\alpha}\dot{\partial}^{\alpha}=C^*_1+C^*_2,
\end{equation}
where, $C_1^*$ and $C_2^*$ denote the Liouville-Hamilton vector fields of $T^*M_1$ and $T^*M_2$, respectively.

Next, the \emph{Liouville 1-form} $\theta$ on $T^*M$ is defined by:
\begin{equation}~\label{l 1 f}
\theta:=\mathbf{p}_ad\mathbf{x}^a=p_idx^i+q_{\alpha}dz^{\alpha}=\theta_1+\theta_2,
\end{equation}
where, $\theta_1$ and $\theta_2$ are Liouville 1-forms of $T^*M_1$ and $T^*M_2$, respectively.

And, the \emph{canonical symplectic structure} $\omega$ on $T^*M$ is defined by $\omega=d\theta$ and has local expression as follows:
\begin{equation}~\label{c s s}
\omega:=d\mathbf{p}_a\wedge d\mathbf{x}^a=dp_i\wedge dx^i+dq_{\alpha}\wedge dz^{\alpha}=\omega_1+\omega_2,
\end{equation}
where $\omega_1$ and $\omega_2$ are canonical symplectic structures of $T^*M_1$ and $T^*M_2$, respectively.

Finally, if the Poisson bracket on the cotangent bundles of $T^*M_1$, $T^*M_2$ and $T^*M$ are denoted by $\{.,.\}_1$, $\{.,.\}_2$ and $\{.,.\}$, respectively, then they are related as follows:
\begin{equation}~\label{poisson}
\{g,h\}=\bar{\partial}^ag\frac{\partial h}{\partial\mathbf{x}^a}-\bar{\partial}^ah\frac{\partial g}{\partial\mathbf{x}^a}=\{g,h\}_1+\{g,h\}_2,
\end{equation}
where $g,h\in C^{\infty}(T^*M)$.

The \emph{Hamilton vector field} of Hamiltonian function $H$ is denoted by $X_H$ and satisfied the following equation
$$\iota_{X_H}\omega=-dH.$$
Let $X_{H_1}$ and $X_{H_2}$ be Hamilton vector fields of the spaces $\mathbb{H}_1^n$ and $\mathbb{H}_2^m$, respectively, then the following theorem gives an expression of $X_H$.
\begin{Theorem}~\label{ham vec}
Suppose that $\mathbb{H}=(M,H)$ is warped product Hamiltonian space defined in~(\ref{w p H}). Then the Hamilton vector field of $\mathbb{H}$ is given by:
$$X_H=X_{H_1}+fX_{H_2}-H_2\frac{\partial f}{\partial x^i}\dot{\partial}^i.$$
\end{Theorem}
\begin{proof}
By definition of Hamilton vector fields i.e. $\iota_{X_{H}}\omega=-dH$. It is a straightforward calculation to complete the prove.
\end{proof}

\section{Nonlinear Connection on Warped Product Hamiltonian Space}
For Hamiltonian spaces $\mathbb{H}_1^n$ and $\mathbb{H}_2^m$, the equations
\begin{equation}~\label{metric1}
\left\{
\begin{array}{l}
g^{ij}=\frac{1}{2}\dot{\partial}^i\dot{\partial}^jH_1, \cr g^{\alpha\beta}=\frac{1}{2}\dot{\partial}^{\alpha}\dot{\partial}^{\beta}H_2.
\end{array}
\right.
\end{equation}
define fundamental tensors of spaces $\mathbb{H}_1^n$ and $\mathbb{H}_2^m$, respectively. The fundamental tensor of warped product Hamiltonian space $(M,H)$ is given by:
\begin{equation}~\label{metric2}
(g^{ab})=\left(\frac{1}{2}\bar{\partial}^a\bar{\partial}^bH\right)=\left(%
\begin{array}{cc}
  g^{ij} & 0 \\
  0 & fg^{\alpha\beta}
  \end{array}%
\right).
\end{equation}
Now, it is easy to check that $(M,H)$ is a Hamilton space as well. By the definition of the canonical nonlinear connections of a Hamiltonian space presented in~\cite{miron}, the canonical nonlinear connections of $\mathbb{H}_1^n$, $\mathbb{H}_2^m$ and $\mathbb{H}$, respectively, are obtained as follows
\begin{equation}~\label{nonlinear}
\left\{
\begin{array}{l}
N_{ij}=\frac{1}{4}\{g_{ij},H_1\}-\frac{1}{4}\left(g_{ik}\frac{\partial^2H_1}{\partial p_k\partial x^j}+g_{jk}\frac{\partial^2H_1}{\partial p_k\partial x^i}\right),\cr
N_{\alpha\beta}=\frac{1}{4}\{g_{\alpha\beta},H_2\}-\frac{1}{4}\left(g_{\alpha\gamma}\frac{\partial^2H_2}{\partial q_{\gamma}\partial z^{\beta}}+g_{\beta\gamma}\frac{\partial^2H_2}{\partial q_{\gamma}\partial z^{\alpha}}\right),\cr
\bar{N}_{ab}=\frac{1}{4}\{g_{ab},H\}-\frac{1}{4}\left(g_{ac}\frac{\partial^2H}
{\partial\mathbf{p}_c\partial\mathbf{x}^b}+g_{bc}\frac{\partial^2H}{\partial\mathbf{p}_c\partial\mathbf{x}^a}\right),
\end{array}
\right.
\end{equation}
where $(g_{ij})$, $(g_{\alpha\beta})$ and $(g_{ab})$ are the inverse matrices of $(g^{ij})$, $(g^{\alpha\beta})$ and $(g^{ab})$, respectively. The relation of the nonlinear connections $\bar{N}_{ab}$ of the Hamiltonian space $\mathbb{H}$ and those of $\mathbb{H}_1^n$ and $\mathbb{H}_2^m$ are given by
\begin{equation}~\label{nonlinear1}
\left\{
\begin{array}{l}
\bar{N}_{ij}=N_{ij}+\frac{1}{4}\dot{\partial}^kg_{ij}\frac{\partial f}{\partial x^k}H_2,\cr
\bar{N}_{\alpha\beta}:=\bar{N}_{(\alpha+n)(\beta+n)}=N_{\alpha\beta}-\frac{1}{4f^2}g_{\alpha\beta}\dot{\partial}^kH_1\frac{\partial f}{\partial x^k},\cr
\bar{N}_{i\alpha}:=\bar{N}_{i(\alpha+n)}=-\frac{1}{4f}g_{\alpha\beta}\dot{\partial}^{\beta}H_2\frac{\partial f}{\partial x^i}.
\end{array}
\right.
\end{equation}
Let $\pi$ be the projection map
$$\pi:=(\pi_1,\pi_2):T^*M_1\oplus T^*M_2\longrightarrow M_1\times M_2.$$
Then, the kernel of $\pi_*$ is known as vertical bundle on $T^*M$ and denoted by $VT^*M$. The local sections of $VT^*M$ are given by
$$\{\frac{\partial}{\partial p_1},...,\frac{\partial}{\partial p_n},\frac{\partial}{\partial q_1},...,\frac{\partial}{\partial q_m}\}.$$
Using the nonlinear connections $\bar{N}_{ij}$, $\bar{N}_{i\alpha}$ and $\bar{N}_{\alpha\beta}$ to define the nonholomorphic vector fields as follows
\begin{equation}~\label{basis}
\left\{
\begin{array}{l}
\frac{\delta^*}{\delta^*x^i}:=\frac{\delta^*}{\delta^*\mathbf{x}^i}=\frac{\partial}{\partial x^i}+\bar{N}_{ij}\dot{\partial}^j+\bar{N}_{i\alpha}\dot{\partial}^{\alpha},\cr
\frac{\delta^*}{\delta^*z^{\alpha}}:=\frac{\delta^*}{\delta^*\mathbf{x}^{\alpha+n}}=\frac{\partial}{\partial z^{\alpha}}+\bar{N}_{\alpha i}\dot{\partial}^i+\bar{N}_{\alpha\beta}\dot{\partial}^{\beta},
\end{array}
\right.
\end{equation}
which make the warped horizontal distribution on $T^*M$ denoted by $HT^*M$. The dual 1-forms of these local vector fields are given by
\begin{equation}~\label{basis1}
\left\{
\begin{array}{l}
d\mathbf{x}^a=dx^i\delta_i^a+dz^{\alpha}\delta_{\alpha+n}^a,\cr
\delta^*p_i:=\delta\mathbf{p}_i=dp_i-\bar{N}_{ij}dx^j-\bar{N}_{i\alpha}dz^{\alpha},\cr
\delta^*q_{\alpha}:=\delta\mathbf{p}_{\alpha+n}=dq_{\alpha}-\bar{N}_{\alpha i}dx^i-\bar{N}_{\alpha\beta}dz^{\beta}.
\end{array}
\right.
\end{equation}
Moreover, the Sasakian metric $G$ on $T^*M$ of the Hamiltonian structure $H$ is defined by
\begin{equation}~\label{metric3}
G=g_{ij}dx^i\otimes dx^j+\frac{g_{\alpha\beta}}{f}dz^{\alpha}\otimes dz^{\beta}+g^{ij}\delta^*p_i\otimes\delta^*p_j+fg^{\alpha\beta}\delta^*q_{\alpha}\otimes\delta^*q_{\beta}.
\end{equation}

\section{The Levi-Civita Connection of Metric $G$}
The Lie brackets of the local vector fields given in previous section are presented as follows
\begin{equation}~\label{lie brac1}
\left\{
\begin{array}{l}
[\frac{\delta^*}{\delta^*x^i},\frac{\delta^*}{\delta^*x^j}]=\mathbf{R}_{ijk}\dot{\partial}^k
+\mathbf{R}_{ij\alpha}\dot{\partial}^{\alpha},\cr
[\frac{\delta^*}{\delta^*x^i},\frac{\delta^*}{\delta^*z^{\alpha}}]=\mathbf{R}_{i\alpha j}\dot{\partial}^j+\mathbf{R}_{i\alpha\beta}\dot{\partial}^{\beta},\cr
[\frac{\delta^*}{\delta^*z^{\alpha}},\frac{\delta^*}{\delta^*z^{\beta}}]=\mathbf{R}_{\alpha\beta i}\dot{\partial}^i+\mathbf{R}_{\alpha\beta\gamma}\dot{\partial}^{\gamma},
\end{array}
\right.
\end{equation}
where
\begin{equation}~\label{R}
\left\{
\begin{array}{l}
\mathbf{R}_{ijk}=\frac{\delta^*\bar{N}_{jk}}{\delta^*x^i}-\frac{\delta^*\bar{N}_{ik}}{\delta^*x^j}\ ,\ \ \ \mathbf{R}_{ij\alpha}=\frac{\delta^*\bar{N}_{j\alpha}}{\delta^*x^i}-\frac{\delta^*\bar{N}_{i\alpha}}{\delta^*x^j},\cr
\mathbf{R}_{i\alpha k}=\frac{\delta^*\bar{N}_{\alpha k}}{\delta^*x^i}-\frac{\delta^*\bar{N}_{ik}}{\delta^*z^{\alpha}}\ ,\ \ \ \mathbf{R}_{i\alpha\beta}=\frac{\delta^*\bar{N}_{\alpha\beta}}{\delta^*x^i}-\frac{\delta^*\bar{N}_{i\beta}}{\delta^*z^{\alpha}},\cr
\mathbf{R}_{\alpha\beta k}=\frac{\delta^*\bar{N}_{\beta k}}{\delta^*z^{\alpha}}-\frac{\delta^*\bar{N}_{\alpha k}}{\delta^*z^{\beta}}\ ,\ \ \ \mathbf{R}_{\alpha\beta\gamma}=\frac{\delta^*\bar{N}_{\beta\gamma}}{\delta^*z^{\alpha}}
-\frac{\delta^*\bar{N}_{\alpha\gamma}}{\delta^*z^{\beta}}.
\end{array}
\right.
\end{equation}
The components $\mathbf{R}_{abc}$ are called \emph{curvature tensors} of nonlinear connection $\bar{N}_{ab}$ and they are skew-symmetric with respect to the indices $a$ and $b$. Moreover
\begin{equation}~\label{lie brac2}
\left\{
\begin{array}{l}
[\dot{\partial}^i,\frac{\delta^*}{\delta^*x^j}]=\dot{\partial}^i(\bar{N}_{jk})\dot{\partial}^k,\cr
[\dot{\partial}^{\alpha},\frac{\delta^*}{\delta^*x^i}]=\dot{\partial}^{\alpha}(\bar{N}_{ik})\dot{\partial}^k
+\dot{\partial}^{\alpha}(\bar{N}_{i\beta})\dot{\partial}^{\beta},\cr
[\dot{\partial}^i,\frac{\delta^*}{\delta^*z^{\alpha}}]=\dot{\partial}^i(\bar{N}_{\alpha\beta})\dot{\partial}^{\beta},\cr
[\dot{\partial}^{\alpha},\frac{\delta^*}{\delta^*z^{\beta}}]=\dot{\partial}^{\alpha}(\bar{N}_{\beta k})\dot{\partial}^k+\dot{\partial}^{\alpha}(\bar{N}_{\beta\gamma})\dot{\partial}^{\gamma}.
\end{array}
\right.
\end{equation}
Let $\nabla$ be the Levi-Civita connection on $(T^*M,G)$ which is given by:
\begin{equation}~\label{levi-civita}
\begin{array}{l}
2G(\nabla_XY,Z)=XG(Y,Z)+YG(X,Z)-ZG(X,Y)\cr
\hspace{2.7cm}-G([X,Z],Y)-G([Y,Z],X)+G([X,Y],Z),
\end{array}
\end{equation}
for any $X,Y,Z\in\Gamma(TT^*M)$. Then, the components of $\nabla$ are given by:

\begin{equation}~\label{levi-civita1}
\left\{
\begin{array}{l}
\nabla_{\frac{\delta^*}{\delta^*x^i}}\frac{\delta^*}{\delta^*x^j}=\Gamma_{ij}^k\frac{\delta^*}{\delta^*x^k}-
\frac{f}{2}\bar{N}_{\alpha k}g_{ij}^kg^{\alpha\beta}\frac{\delta^*}{\delta^*z^{\beta}}+
\frac{1}{2}g_{ijk}\dot{\partial}^k+\frac{1}{2}\mathbf{R}_{ija}\bar{\partial}^a\cr

\nabla_{\frac{\delta^*}{\delta^*x^i}}\frac{\delta^*}{\delta^*z^{\alpha}}=\nabla_{\frac{\delta^*}{\delta^*z^{\alpha}}}
\frac{\delta^*}{\delta^*x^i}+\mathbf{R}_{i\alpha a}\bar{\partial}^a=-\frac{1}{2}\bar{N}_{\alpha j}g_i^{jk}\frac{\delta^*}{\delta^*x^k}\cr
\hspace{2cm}+\frac{1}{2}(\frac{\partial\ln f}{\partial x^i}\delta_{\alpha}^{\gamma}-\bar{N}_{i\beta}g_{\alpha}^{\beta\gamma})\frac{\delta^*}{\delta^*z^{\gamma}}+\frac{1}{2}\mathbf{R}_{i\alpha a}\bar{\partial}^a\cr

\nabla_{\frac{\delta^*}{\delta^*z^{\alpha}}}\frac{\delta^*}{\delta^*z^{\beta}}=-\frac{1}{2}
\frac{\delta^*fg_{\alpha\beta}}{\delta^*x^i}g^{ij}\frac{\delta^*}{\delta^*x^j}+\Gamma_{\alpha\beta}^{\gamma}
\frac{\delta^*}{\delta^*z^{\gamma}}+\frac{1}{2f^2}g_{\alpha\beta\lambda}
\dot{\partial}^{\lambda}\cr
\hspace{2cm}+\frac{1}{2}\mathbf{R}_{\alpha\beta a}\bar{\partial}^a
\end{array}
\right.
\end{equation}
\begin{equation}~\label{levi-civita2}
\left\{
\begin{array}{l}
\nabla_{\dot{\partial}^i}\dot{\partial}^{\alpha}=\nabla_{\dot{\partial}^{\alpha}}\dot{\partial}^i=\frac{1}{8}
\dot{\partial}^{\alpha}H_2g^{ikh}\frac{\partial f}{\partial x^k}\frac{\delta^*}{\delta^*x^h}\cr
\hspace{1.5cm}-\frac{1}{2}(f^2\dot{\partial}^i(\bar{N}_{\beta\gamma})g^{\gamma\alpha}g^{\beta\lambda}+
f\dot{\partial}^{\alpha}(\bar{N}_{\beta k})g^{ki}g^{\beta\lambda})\frac{\delta^*}{\delta^*z^{\lambda}}\cr

\nabla_{\dot{\partial}^i}\dot{\partial}^j=-\frac{1}{2}(\frac{\delta^*g^{ij}}{\delta^*x^k}+
\dot{\partial}^i(\bar{N}_{kt})g^{tj}+\dot{\partial}^j(\bar{N}_{kt})g^{ti})g^{kh}\frac{\delta^*}{\delta^*x^h}\cr
\hspace{1.5cm}+\frac{1}{8}\dot{\partial}^{\beta}H_2g^{ijk}\frac{\partial f}{\partial x^k}\frac{\delta^*}{\delta^*z^{\beta}}+\frac{1}{2}g_k^{ij}\dot{\partial}^k\cr

\nabla_{\dot{\partial}^{\alpha}}\dot{\partial}^{\beta}=-\frac{1}{2}(\frac{\delta^*fg^{\alpha\beta}}{\delta^*x^k}+
f\dot{\partial}^{\alpha}(\bar{N}_{k\gamma})g^{\gamma\beta}+f\dot{\partial}^{\beta}(\bar{N}_{k\gamma})g^{\gamma\alpha}
)g^{kh}\frac{\delta^*}{\delta^*x^h}\cr
\hspace{1.5cm}-\frac{f^2}{2}(\frac{\delta^*g^{\alpha\beta}}{\delta^*z^{\gamma}}+\dot{\partial}^{\alpha}(\bar{N}_{\gamma\theta})
g^{\theta\beta}+\dot{\partial}^{\beta}(\bar{N}_{\gamma\theta})g^{\theta\alpha})g^{\gamma\lambda}
\frac{\delta^*}{\delta^*z^{\lambda}}+\frac{1}{2}g_{\gamma}^{\alpha\beta}\dot{\partial}^{\gamma}
\end{array}
\right.
\end{equation}
\begin{equation}~\label{levi-civita3}
\left\{
\begin{array}{l}
\nabla_{\frac{\delta^*}{\delta^*x^i}}\dot{\partial}^j=\nabla_{\dot{\partial}^j}\frac{\delta^*}{\delta^*x^i}-
\dot{\partial}^j(\bar{N}_{ik})\dot{\partial}^k=-\frac{1}{2}\dot{\partial}^j(\bar{N}_{ik})\dot{\partial}^k\cr
\hspace{1.5cm}-\frac{1}{2}(g_i^{jh}+\mathbf{R}_{iks}g^{sj}g^{kh})
\frac{\delta^*}{\delta^*x^h}-\frac{f}{2}\mathbf{R}_{i\alpha k}g^{kj}g^{\alpha\beta}\frac{\delta^*}{\delta^*z^{\beta}}\cr
\hspace{1.5cm}+\frac{1}{2}(\frac{\delta^*g^{jk}}{\delta^*x^i}+\dot{\partial}^k(\bar{N}_{is})g^{sj})g_{kh}\dot{\partial}^h+
\frac{1}{2f}\dot{\partial}^{\alpha}(\bar{N}_{ik})g^{kj}g_{\alpha\beta}\dot{\partial}^{\beta}\cr

\nabla_{\frac{\delta^*}{\delta^*x^i}}\dot{\partial}^{\alpha}=\nabla_{\dot{\partial}^{\alpha}}
\frac{\delta^*}{\delta^*x^i}-\dot{\partial}^{\alpha}(\bar{N}_{ia})\bar{\partial}^a=-\frac{1}{2}\dot{\partial}^{\alpha}(\bar{N}_{ia})\bar{\partial}^a\cr
\hspace{1.5cm}\frac{f}{2}\mathbf{R}_{ki\beta}g^{\beta\alpha}g^{kh}\frac{\delta^*}{\delta^*x^h}+\frac{f^2}{2}\mathbf{R}_{\beta i\gamma}g^{\gamma\alpha}g^{\beta\lambda}\frac{\delta^*}{\delta^*z^{\lambda}}\cr
\hspace{1.5cm}+\frac{1}{2}(\frac{1}{f}\frac{\delta^*fg^{\alpha\beta}}{\delta^*x^i}+\dot{\partial}^{\beta}(\bar{N}_{i\gamma})
g^{\gamma\alpha})g_{\beta\lambda}\dot{\partial}^{\lambda}\cr

\nabla_{\frac{\delta^*}{\delta^*z^{\alpha}}}\dot{\partial}^i=\nabla_{\dot{\partial}^i}
\frac{\delta^*}{\delta^*z^{\alpha}}-\dot{\partial}^i(\bar{N}_{\alpha \beta})\dot{\partial}^{\beta}=-\frac{1}{2}\dot{\partial}^i(\bar{N}_{\alpha\beta})\bar{\partial}^{\beta}\cr
\hspace{1.5cm}\frac{1}{2}\mathbf{R}_{k\alpha s}g^{si}g^{kh}\frac{\delta^*}{\delta^*x^h}+\frac{f}{2}\mathbf{R}_{\beta\alpha k}g^{ki}g^{\beta\gamma}\frac{\delta^*}{\delta^*z^{\gamma}}+\frac{1}{2}\frac{\delta^*g^{ik}}{\delta^*z^{\alpha}}g_{kh}\dot{\partial}^h\cr
\hspace{1.5cm}+\frac{1}{2f}\dot{\partial}^{\beta}(\bar{N}_{\alpha k})g^{ki}g_{\beta\gamma}\dot{\partial}^{\gamma}\cr

\nabla_{\frac{\delta^*}{\delta^*z^{\alpha}}}\dot{\partial}^{\beta}=\nabla_{\dot{\partial}^{\beta}}
\frac{\delta^*}{\delta^*z^{\alpha}}-\dot{\partial}^{\beta}(\bar{N}_{\alpha a})\bar{\partial}^a=-\frac{1}{2}\dot{\partial}^{\beta}(\bar{N}_{\alpha a})\bar{\partial}^a\cr
\hspace{1.5cm}\frac{f}{2}\mathbf{R}_{k\alpha\gamma}g^{\gamma\beta}g^{kh}\frac{\delta^*}{\delta^*x^h}-\frac{1}{2}(
g_{\alpha}^{\beta\lambda}+f^2\mathbf{R}_{\alpha\gamma\theta}g^{\theta\beta}g^{\gamma\lambda})\frac{\delta^*}{\delta^*z^{\lambda}}\cr
\hspace{1.5cm}-\frac{1}{4f}\delta_{\alpha}^{\beta}\frac{\partial f}{\partial x^j}\dot{\partial}^j+\frac{1}{2}
(\frac{\delta^*g^{\beta\gamma}}{\delta^*z^{\alpha}}g_{\gamma\lambda}+\dot{\partial}^{\gamma}(\bar{N}_{\alpha\theta})
g^{\theta\beta}g_{\gamma\lambda})\dot{\partial}^{\lambda}
\end{array}
\right.
\end{equation}
where,
$$g^{abc}=\bar{\partial}^ag^{bc},\ \ g_{abc}=g_{cf}g_{ab}^f=g_{cf}g_{be}g_a^{ef}=g_{cf}g_{be}g_{ad}g^{def}$$
and
$$\Gamma_{ij}^k=\frac{g^{kh}}{2}\left(\frac{\delta^*g_{jh}}{\delta^*x^i}+\frac{\delta^*g_{ih}}{\delta^*x^j}
-\frac{\delta^*g_{ij}}{\delta^*x^h}\right),$$
$$\Gamma_{\alpha\beta}^{\gamma}=\frac{g^{\gamma\lambda}}{2}\left(\frac{\delta^*g_{\beta\lambda}}{\delta^*z^{\alpha}}
+\frac{\delta^*g_{\alpha\lambda}}{\delta^*z^{\beta}}
-\frac{\delta^*g_{\alpha\beta}}{\delta^*z^{\lambda}}\right).$$
\section{Foliations on Warped Product Hamiltonian Spaces}
In this section, we study geometric properties of vertical distribution $VT^*M$ such as being bundle-like with respect to the metric $G$ and being totally geodesic. The conditions which are equivalent to these properties show a close relation between the geometry of the warped Hamiltonian manifold and its base Hamiltonian spaces.

\begin{Theorem}~\label{bundle like}
Let $\mathbb{H}=(M,H)$ be a warped product Hamiltonian space with nonconstant warped function $f$. Then, the warped Sasaki metric $G$ is bundlelike for vertical foliation $VT^*M$ if and only if $(M_1,(g_{ij}))$ and $(M_2,(g_{\alpha\beta}))$ are two pseudo-Riemannian manifolds.
\end{Theorem}
\begin{proof}
With respect to bundle-like condition (see~\cite{beja,moli}), $G$ is bundle-like for $VT^*M$ if and only if:
$$G(\nabla_XY+\nabla_YX,Z)=0 \ \ \ \ \forall X,Y\in\Gamma(HT^*M), Z\in\Gamma(VT^*M),$$
It is equivalent to following equations:
$$G(\nabla_{\frac{\delta^*}{\delta^*x^i}}\frac{\delta^*}{\delta^*x^j}+\nabla_{\frac{\delta^*}{\delta^*x^j}}
\frac{\delta^*}{\delta^*x^i},\dot{\partial}^k)=G(\nabla_{\frac{\delta^*}{\delta^*x^i}}\frac{\delta^*}{\delta^*x^j}
+\nabla_{\frac{\delta^*}{\delta^*x^j}}\frac{\delta^*}{\delta^*x^i},\dot{\partial}^{\alpha})=0,$$
$$G(\nabla_{\frac{\delta^*}{\delta^*u^{\alpha}}}\frac{\delta^*}{\delta^*u^{\beta}}+\nabla_{
\frac{\delta^*}{\delta^*u^{\beta}}}\frac{\delta^*}{\delta^*u^{\alpha}},\dot{\partial}^i)=G(\nabla_{
\frac{\delta^*}{\delta^*u^{\alpha}}}\frac{\delta^*}{\delta^*u^{\beta}}+\nabla_{
\frac{\delta^*}{\delta^*u^{\beta}}}\frac{\delta^*}{\delta^*u^{\alpha}},\dot{\partial}^{\gamma})=0,$$
$$G(\nabla_{\frac{\delta^*}{\delta^*x^i}}\frac{\delta^*}{\delta^*u^{\alpha}}+\nabla_{
\frac{\delta^*}{\delta^*u^{\alpha}}}\frac{\delta^*}{\delta^*x^i},\dot{\partial}^j)=G(\nabla_{\frac{\delta^*}{\delta^*x^i}}\frac{\delta^*}{\delta^*u^{\alpha}}+\nabla_{
\frac{\delta^*}{\delta^*u^{\alpha}}}\frac{\delta^*}{\delta^*x^i},\dot{\partial}^{\beta})=0.$$
By using~(\ref{levi-civita})-(\ref{levi-civita3}), one can obtain that above equations are satisfied if and only if $g_{ijk}=g_{\alpha\beta\gamma}=0$, and this completes the proof.
\end{proof}

\begin{Theorem}~\label{tot. geo.}
Let $\mathbb{H}=(M,H)$ be a warped product Hamiltonian space with nonconstant warped function $f$. Then, $\mathbb{H}=(M,H)$ is a Landsberg-Hamilton space if and only if the vertical foliation $VT^*M$ is totally geodesic.
\end{Theorem}
\begin{proof}
With respect to the definition of Landsberg-Hamilton space~\cite{miron}, $(M,H)$ is a Landsberg-Hamilton space if and only if 
$$g_{ab|_*c}=\frac{\delta^*g_{ab}}{\delta^*\mathbf{x}^c}+g^{bd}\dot{\partial}^a(\bar{N}_{dc})+g^{ad}\dot{\partial}^b(\bar{N}_{dc})=0,$$
By using~(\ref{levi-civita})-(\ref{levi-civita3}), one can check that $$g_{ab|_*c}=0,$$ is satisfied if and only if $VT^*M$ is totally geodesic, and this completes the proof.
\end{proof}

\begin{Theorem}~\label{tot geo}
Let $\mathbb{H}=(M,H)$ be a warped product Hamiltonian space with nonconstant warped function $f$. Then, the horizontal distribution $HT^*M$ is a totally geodesic one if and only if $(M,H)$ is an Euclidean space.
\end{Theorem}
\begin{proof}
Suppose that $HT^*M$ is a totally geodesic distribution, then
$$\nabla_{\frac{\delta^*}{\delta^*x^i}}\frac{\delta^*}{\delta^*x^j},\ \nabla_{\frac{\delta^*}{\delta^*x^i}}\frac{\delta^*}{\delta^*u^{\alpha}},\ \nabla_{
\frac{\delta^*}{\delta^*u^{\alpha}}}\frac{\delta^*}{\delta^*x^i},\ \nabla_{\frac{\delta^*}{\delta^*u^{\alpha}}}\frac{\delta^*}{\delta^*u^{\beta}}\in\Gamma(HT^*M).$$
From Eq.~(\ref{levi-civita1}), above conditions are hold if and only if
$$\mathbf{R}_{abc}=g_{abc}=0,$$
These equations are equivalent to $(M,H)$ is an Euclidean space (the Pseudo-Riemannian space with zero curvature).
\end{proof}

Combining the theorems~\ref{bundle like} and~\ref{tot. geo.}, we have the following corollary.
\begin{cor}
Let the warped product Hamiltonian space $(M,H)$ be a pseudo-Riemannian manifold with nonconstant warped function $f$, then the vertical distribution $VT^*M$ is totally geodesic and metric $G$ is bundle-like for $VT^*M$.
\end{cor}

\bibliographystyle{elsarticle-num}

\end{document}